\documentclass{amsart}

\usepackage{ae}
\usepackage[all]{xy}
\usepackage{graphicx,amsfonts,amssymb,amsmath}
\usepackage{eurofont}
\usepackage{natbib}

\usepackage[latin1]{inputenc}
\usepackage[T1]{fontenc}
\usepackage[english,francais]{babel}
\usepackage[cyr]{aeguill}

\usepackage[OT2,T1]{fontenc}
\DeclareSymbolFont{cyrletters}{OT2}{wncyr}{m}{n}
\DeclareMathSymbol{\sha}{\mathalpha}{cyrletters}{"58}
\input cyracc.def

 \newtheorem{thm}{Théorème}[section]
 \newtheorem*{thmp}{Théorème principal}
 \newtheorem{conjecture}[thm]{Conjecture}
 \newtheorem{cor}[thm]{Corollaire}
 \newtheorem{lem}[thm]{Lemme}
 \newtheorem{prop}[thm]{Proposition}
 \theoremstyle{definition}
 \newtheorem{defn}[thm]{Définition}
 \theoremstyle{remark}
 
 \theoremstyle{remark}
 \newtheorem{rem}[thm]{\textbf{Remarque}}
 \numberwithin{equation}{subsection}

 \newcommand{\To}{\longrightarrow}

 \newcommand{\Q}{\mathbb{Q}}
 \newcommand{\Z}{\mathbb{Z}}

\usepackage[dvipdfm]{hyperref}

\begin{document}

\title[Principe local-global pour les zéro-cycles - I]
 {Principe local-global pour les zéro-cycles sur certaines fibrations au-dessus d'une courbe: I}

\author{ Yongqi LIANG  }

\address{Yongqi LIANG \newline
Département de Mathématiques, \newline Bâtiment 425,\newline Université  Paris-sud 11,\newline  F-91405 Orsay,\newline
 France}

\email{yongqi.liang@math.u-psud.fr}

\thanks{\textit{Mots clés} : zéro-cycle de degré $1$, principe de Hasse, approximation faible,
obstruction de Brauer-Manin}

\thanks{\textit{Classification AMS} : 14G25 (11G35, 14D10)}

\date{\today.}



\begin{abstract}
Soit $X$ une variété projective lisse sur un corps de nombres,
fibrée au-dessus d'une courbe $C,$ à fibres géométriquement
intègres. On démontre que, en supposant la finitude de
$\sha(Jac(C)),$ si les fibres au-dessus d'un sous-ensemble
hilbertien généralisé satisfont le principe de Hasse (resp.
l'approximation faible), alors l'obstruction de Brauer-Manin
provenant de la courbe en bas est la seule au principe de Hasse
(resp. à l'approximation faible) pour les zéro-cycles de degré
$1$ sur $X.$ Ceci est appliqué à l'exemple récent
de Poonen.
\end{abstract}

\maketitle
\tableofcontents

\section*{Introduction}
\label{intro}

Soit $X$ une variété (\textit{i.e.} un schéma séparé de type fini)
projective lisse géométriquement intègre sur un corps parfait
$k.$ On note $Br(X)=H^2_{\mbox{\scriptsize\'et}}(X,\mathbb{G}_m)$ le groupe de Brauer cohomologique de $X.$
Pour tout point fermé $P$ de corps résiduel $k(P),$ on note $b(P)$ l'image d'un élément
$b\in Br(X)$ dans $Br(k(P)).$ On note $Z_0(X)$ le groupe des zéro-cycles de $X$ et $CH_0(X)$ le groupe de Chow des zéro-cycles
de $X.$
On peut ainsi définir un accouplement
\begin{equation*}
    \begin{array}{rcccl}
     \langle\cdot,\cdot\rangle_k:   Z_0(X) & \times & Br(X) & \to & Br(k),\\
          (\mbox{ }\sum_Pn_PP &,& b\mbox{ }) & \mapsto & \sum_Pcores_{k(P)/k}(b(P)),\\
    \end{array}
\end{equation*}
qui se factorise à travers $CH_0(X).$

Lorsque $k$ est un corps de nombres, on définit l'\emph{accouplement de Brauer-Manin} pour les
zéro-cycles:
\begin{equation*}
    \begin{array}{rcccl}
     \langle\cdot,\cdot \rangle _k:\prod_{v\in\Omega_k}Z_0(X_v) & \times & Br(X) & \to & \mathbb{Q}/\mathbb{Z},\\
         (\mbox{ }\{z_v\}_{v\in\Omega_k} & , & b\mbox{ }) & \mapsto & \sum_{v\in\Omega_k}inv_v(\langle z_v,b \rangle _{k_v}),\\
    \end{array}
\end{equation*}
où $\Omega_k$ est l'ensemble des places de $k,$
$inv_v:Br(k_v)\hookrightarrow\mathbb{Q}/\mathbb{Z}$
est l'invariant local en $v,$ et $X_v=X\times_kk_v.$
Grâce à la théorie du corps de classes global, l'image de $Z_0(X)$ dans $\prod_{v\in\Omega_k}Z_0(X_v)$
est orthogonale au groupe $Br(X).$ La notion d'obstruction de Brauer-Manin au principe de Hasse pour les zéro-cycles
sur $X$ est alors définie.
Dans son article \cite{CT99HP0-cyc}, les conjectures suivantes sont énoncées par Colliot-Thélène
pour une variété $X$ supposée projective lisse et géométriquement intègre de dimension $d$ sur un corps de nombres $k.$

\begin{conjecture}\label{conj1}
L'obstruction de Brauer-Manin au principe de Hasse pour les zéro-cycles de degré $1$ sur $X$ est la seule.
Autrement dit, s'il existe une famille de zéro-cycles $\{z_v\}\in\prod_{v\in\Omega_k}Z_0(X_v)$ de degré $1$ orthogonale au groupe
de Brauer $Br(X),$ alors il existe un zéro-cycle global de degré $1$ sur $X.$
\end{conjecture}

\begin{conjecture}\label{conj2}
Soit $\delta$ un entier.
Soit $\{z_v\}\in\prod_{v\in\Omega_k}Z_0(X_v)$ une famille de zéro-cycles de degré $\delta$ orthogonale au groupe
de Brauer $Br(X).$ Alors, pour tout entier strictement positif $m,$ il existe un zéro-cycle $z_m\in Z_0(X)$ de degré $\delta$
tel que, en chaque place $v,$ les zéro-cycles $z_m$ et $z_v$ ont la même image
par l'application cycle $CH_0(X_v)\to\tilde{H}^{2d}(X_v,\Z/m\Z(d)).$
\end{conjecture}

La conjecture \ref{conj1} pour une courbe est premièrement montrée par Saito \cite{Saito}
(en supposant la finitude du groupe $\sha(Jac(C))$).
Pour les variétés de dimension supérieure, peu de résultats en toute généralité sont établis.
Quand $X$ admet une structure de fibration au-dessus d'une courbe,
Colliot-Thélène, Skorobogatov, et Swinnerton-Dyer ont montré, dans \cite{CT-Sk-SD},
la conjecture \ref{conj1} pour une fibration $X\To \mathbb{P}^1$ à fibres satisfaisant le principe de Hasse pour les
zéro-cycles de degré $1,$ en faisant une hypothèse
(H fibre) sur les fibres singulières; dans son article \cite{Frossard}, Frossard
a montré la conjecture \ref{conj1} pour une fibration $X\To C$
en variétés de Severi-Brauer d'indice sans facteur carré
au-dessus d'une courbe $C$ (en supposant la finitude du groupe $\sha(Jac(C))$)
si $k$ est totalement imaginaire.
Des variantes de la conjecture \ref{conj2} sont également discutées dans ces articles.

Dans le présent travail, notre résultat principal est le suivant (voir le texte
(\ref{notionAF} et \ref{hilbertien})
pour les notions de  \og \emph{approximation faible au niveau de la cohomologie} \fg{}
et \og \emph{sous-ensemble hilbertien généralisé} \fg{},
et pour les énoncés plus précis avec plus de détails).

\begin{thmp}[Théorèmes \ref{thm4}, \ref{thm5}]\label{mainthm1}\ \\
Soit $\pi:X\To C$ une fibration à fibres géométriquement intègres avec $\sha(Jac(C))$ fini.
Soit $\textsf{Hil}$ un sous-ensemble hilbertien généralisé de $C.$
Supposons que pour tout point fermé $\theta\in \textsf{Hil},$
la fibre $X_\theta$ satisfait le principe
de Hasse (resp. l'approximation faible au niveau de la cohomologie aux places finies)
pour les zéro-cycles de degré $1.$

Alors, l'obstruction de Brauer-Manin associée au sous-groupe $\pi^*Br(C)\subset Br(X)$
au principe de Hasse (resp. à l'approximation faible au niveau de la cohomologie aux places finies) est la seule
pour les zéro-cycles de degré $1$ sur $X.$
\end{thmp}

Pour une fibration au-dessus d'une courbe vérifiant les hypothèses du théorème,
ce résultat démontre la conjecture \ref{conj1},
et il démontre une version faible de la conjecture \ref{conj2}.
Dans le cas particulier où toutes les fibres sont géométriquement intègres,
le théorème 0.3 de Frossard \cite{Frossard} est généralisé par le théorème principal ici.
Avec une hypothèse plus forte sur les fibres, le théorème principal étend le théorème 4.1
de Colliot-Thélène/Skorobogatov/Swinnerton-Dyer \cite{CT-Sk-SD} au cas où la courbe en bas $C$ est
de genre quelconque (avec $\sha(Jac(C))$ fini) au lieu de $\mathbb{P}^1.$
Le théorème principal est aussi une version pour les zéro-cycles de degré $1$ sur $X\To C$
parallèle au théorème 1 de Skorobogatov \cite{Skorobogatov}.

Comme application, on considère certaines fibrations en surfaces de Châtelet au-dessus d'une courbe sur un corps de nombres $k.$
Dans \cite{Poonen}, Poonen a trouvé une variété $X$ de dimension $3,$ qui est une fibration en surfaces de
Châtelet au-dessus d'une courbe, telle que l'obstruction de Brauer-Manin\footnote{De plus, il n'y a pas
d'obstruction de Brauer-Manin appliquée aux revêtements étales, \textit{cf.} \cite{Poonen} pour plus de détails.}
au principe de Hasse pour les
points rationnels disparaît mais il n'existe pas de point $k$-rationnel. Dans ce cas-là, \textit{a fortiori},
il n'y a pas d'obstruction de Brauer-Manin au principe de Hasse pour les zéro-cycles de degré $1.$ Récemment,
dans son article \cite{CTsurPoonen}, Colliot-Thélène a démontré l'existence d'un zéro-cycle de degré $1$
sur les solides de Poonen. Le théorème principal ici retrouve le résultat de Colliot-Thélène
si l'on suppose la finitude de $\sha(Jac(C)),$ et de plus affirme
qu'il y a beaucoup de zéro-cycles de degré $1$ au sens de l'approximation faible.

Dans son travail en cours \cite{Wittenberg}, Wittenberg démontre aussi un théorème parallèle à
\ref{thm4} et \ref{thm5} sous l'hypothèse (H fibre), qui est plus faible que ``toute fibre est géométriquement
intègre'', mais là il doit faire les hypothèses arithmétiques sur toutes les
fibres lisses, ceci l'empêche d'appliquer à l'exemple de Poonen. On va généraliser les résultats de Wittenberg
et de ce travail dans l'article en progrès de l'auteur \cite{Liang2}, certaines fibrations au-dessus de $\mathbb{P}^n$
sont aussi considérées.

\medskip
Le texte est organisé comme suit.
Dans la première partie de cet article, on introduit les notions nécessaires pour ce travail. Ensuite, on explique
en détails un résultat important sur
l'approximation pour les zéro-cycles sur une courbe quelconque.
Dans la seconde partie, on énonce les théorèmes plus précis \ref{thm4}, \ref{thm5},
et on discute en détails quelques applications du théorème principal, y compris l'application aux solides
de Poonen.
Dans la dernière partie, on montre ces théorèmes.
Premièrement, on rappelle des lemmes de déplacement et l'estimation de Lang-Weil.
Ensuite, on établit un lemme d'irréductibilité de Hilbert pour certains zéro-cycles effectifs.
Enfin, on montre une proposition clé, ainsi que les théorèmes principaux avec l'aide de l'application de Gysin.

\smallskip
\small
\noindent \textbf{Remerciements.}
Ce travail s'est accompli sous la direction de D.Harari.
Je tiens à le remercier de m'avoir proposé ce sujet,
pour ses nombreuses discussions très utiles pendant la préparation de ce travail, et pour son aide pour le français.
Je remercie également J.-L. Colliot-Thélène de m'avoir bien expliqué tous les détails de son article
\cite{CT99HP0-cyc}, pour m'avoir permis de présenter ici sa preuve du théorème \ref{CTresult}, et
pour ses suggestions précieuses.
Je remercie O. Wittenberg de sa patiente explication de son travail en cours \cite{Wittenberg},
et pour ses commentaires pertinents après la lecture de la première version de ce travail.

\normalsize


\section{Préliminaires}

On introduit les notions nécessaires pour cet article, plus particulièrement la notion de
sous-ensemble hilbertien généralisé est introduite et
la notion d'approximation faible/forte
au niveau de la cohomologie est discutée en détails. Ensuite, on explique un résultat sur les courbes, qui sera
utile pour la preuve des théorèmes principaux.


\subsection{Notions générales}
Dans ce travail, $k$ est toujours un corps de nombres,
on note $\Omega_k$ (ou simplement $\Omega$) l'ensemble des places de $k,$
on note $\Omega^\textmd{f}$ (resp. $\Omega^\infty$) le sous-ensemble des places non archimédiennes (resp. archimédiennes).
On note $k_v$ le complété $v$-adique de $k$ pour chaque $v\in\Omega.$ On fixe une clôture algébrique $\bar{k}$ (resp. $\bar{k}_v$)
de $k$ (resp. de $k_v$), telle que le diagramme suivant soit commutatif pour toute place $v.$
\SelectTips{eu}{12}$$\xymatrix@C=20pt @R=10pt{
\bar{k}\ar[r]&\bar{k}_v
\\ k\ar[r]\ar[u]&k_v\ar[u]
}$$

On ne considère que les \emph{fibrations} $\pi:X\To C$ au-dessus une courbe, \textit{i.e.}
$X$ et $C$ sont des variétés projectives lisses et géométriquement intègres sur $k,$
$dim(C)=1,$ le morphisme $\pi$ est non constant (donc plat et surjectif), et on suppose toujours que\ \\
$\star$ la fibre générique $X_\eta$ est une variété géométriquement intègre sur $K,$
où $\eta=Spec(K)$ est le point générique de $C$ avec $K=k(C)$ le corps des fonctions de $C.$

Il existe un ouvert de Zariski non vide $U$ de $C$ tel que pour
tout point $\theta\in U$ la fibre $X_\theta$ soit lisse et géométriquement intègre sur $k(\theta),$
on note $D=C\setminus U$ son complémentaire, $D$ est alors un ensemble fini de points fermés.
L'expression ``presque tout" signifie ``tout à l'exception d'un nombre fini".

\begin{defn}
(1)Soit $z=\sum n_iP_i\in Z_0(C)$ un zéro-cycle de $C$ (avec les points fermés $P_i$ différents deux à deux).
On dit qu'il est \emph{séparable} si $n_i\in\{0,1,-1\}$ pour tout $i,$
il est \emph{déployé} (relativement à une fibration fixée $\pi:X\To C$)
s'il existe un $k(P_i)$-point rationnel sur la fibre $X_{P_i}$
pour tout $i.$
(On pourrait dire que $z$ est déployé lorsqu'il existe un zéro-cycle de degré $1$ sur chaque fibre $X_{P_i}/k(P_i),$
mais dans le présent article on trouvera toujours qu'il existe vraiment un point rationnel dans chaque fibre.)

(2)Soit $y\in Z_0(C)$ un zéro-cycle,
on pose $$Z_0(X/y)=\{z\in Z_0(X); \pi_*(z)\sim y\},$$ où dans ce travail $\sim$ dénote l'équivalence rationnelle.

(3)Étant donné $P$ un point fermé de $X_v=X\times_kk_v,$ on fixe un $k_v$-plongement $k_v(P)\To \bar{k}_v,$
$P$ est vu comme un point $k_v(P)$-rationnel de $X_v.$
On dit qu'un point fermé $Q$ de $X_v$ est \emph{suffisamment proche}  de $P$ (par rapport à un voisinage $U_P$ de $P$ dans
l'espace topologique $X_v(k_v(P))$), si $Q$ a corps résiduel $k_v(Q)=k_v(P)$ et si l'on peut choisir un
$k_v$-plongement $k_v(Q)\To \bar{k}_v$ tel que $Q,$ vu comme un $k_v(Q)$-point rationnel de $X_v,$ soit contenu dans $U_P.$
En étendant $\mathbb{Z}$-linéairement, cela a un sens de dire que $z'_v\in Z_0(X_v)$ est suffisamment proche de $z_v\in Z_0(X_v)$
(par rapport à
un système de voisinages des points qui apparaissent dans le support de $z_v$), et on note $z'_v\approx z_v$ dans ce cas.
En particulier, $deg(z'_v)=deg(z_v)$ si $z'_v\approx z_v.$
D'après la continuité de l'accouplement de Brauer-Manin, \textit{cf.} \cite{Ducros} Partie II (0.31),
pour un sous-ensemble \emph{fini} $B\subset Br(X_v),$ on a
$\langle z'_v,b \rangle _v= \langle z_v,b \rangle _v\in Br(k_v)$ pour tout $b\in B$ si
$z'_v\approx z_v$ (par rapport à $B$).
\end{defn}


\subsection{Sous-ensembles hilbertiens généralisés}
On généralise la notion de sous-ensemble hilbertien des points rationnels à des points fermés.

\begin{defn}\label{hilbertien}
Soit $X$ une variété géométriquement intègre sur un corps de nombres $k.$
Un sous-ensemble $\textsf{Hil}\subset X$ de points fermés
de $X$ est dit un \emph{sous-ensemble hilbertien généralisé} s'il existe un morphisme étale fini
$Z\buildrel\rho\over\To U\subset X$ avec $U$ un ouvert non vide de $X$
et $Z$ intègre tel que $\textsf{Hil}$ soit l'ensemble des points fermés
$M$ de $U$ pour lesquels $\rho^{-1}(M)$ est connexe.
\end{defn}

Si $\textsf{Hil}$ est un sous-ensemble hilbertien généralisé de $X,$ $\textsf{Hil}\cap X(k)$
est alors un sous-ensemble hilbertien (classique), \textit{cf.} \cite{Harari}, 3.2.
Les points fermés d'un ouvert non vide de $X$ forment
un sous-ensemble hilbertien généralisé de $X.$

Si $X$ est une variété normale, soit $\textsf{Hil}_i$ ($i=1,2$) un sous-ensemble hilbertien généralisé de $X$ défini par
$Z_i\To U_i\subset X,$ on note $K$ (resp. $K_i$) le corps des fonctions de $X$ (resp. de $Z_i$) et $L=K_1K_2$ la composition
de $K_1$ et $K_2$ dans $\bar{K}.$  L'extension $L$ définit alors un revêtement fini $Z\To X,$
qui est étale au-dessus d'un ouvert non vide $U$ de $U_1\cap U_2\subset X,$
ainsi définit un sous-ensemble hilbertien généralisé
$\textsf{Hil}\subset \textsf{Hil}_1\cap \textsf{Hil}_2.$


\subsection{Approximation faible/forte au niveau de la cohomologie}
On rappelle la discussion dans \cite{CT99HP0-cyc}.
Soit $X$ une variété propre géométriquement intègre.
On a un complexe de groupes abéliens
\SelectTips{eu}{12}$$\xymatrix@C=20pt @R=10pt{
CH_0(X)\ar[r]^-{\psi}&\prod_{v\in\Omega}CH_0(X_v)\ar[r]^-{\gamma}&Hom(Br(X),\mathbb{Q}/\mathbb{Z})
}$$
où l'application $\gamma$ est induite par l'accouplement de Brauer-Manin
$$\prod_{v\in\Omega}CH_0(X_v)\times Br(X)\To \mathbb{Q}/\mathbb{Z}.$$

Pour un entier strictement positif $m,$
on a l'application cycle $\phi_m:CH_0(X)\to H^{2d}(X,\mathbb{Z}/m\mathbb{Z}(d))$ qui se factorise
à travers $CH_0(X)/m,$ où $d=dim(X).$

Pour chaque place $v,$ on a la cohomologie modifiée $\tilde{H}^{2d}(X_v,\mathbb{Z}/m\mathbb{Z}(d))$
et une application canonique $H^{2d}(X_v,\mathbb{Z}/m\mathbb{Z}(d))\To\tilde{H}^{2d}(X_v,\mathbb{Z}/m\mathbb{Z}(d)),$
qui est un isomorphisme (resp. une flèche nulle) si $v$ est une place non archimédienne (resp. une place complexe),
voir \cite{Saito0} et \cite{CT99HP0-cyc} \S 1 pour plus d'informations. On note
$\phi_{m,v}$ le composé de cette flèche et l'application cycle pour $X_v.$

Pour un entier strictement positif $m,$
on trouve le diagramme commutatif suivant, dont la première ligne est un complexe
et la deuxième ligne est une suite exacte de groupes abéliens (\textit{cf.} \cite{Saito0}, \cite{CT99HP0-cyc} \S 1).
\SelectTips{eu}{12}$$\xymatrix@C=20pt @R=20pt{
CH_0(X)\ar[d]^{\phi_m}\ar[r]^-{\psi}&\prod_{v\in\Omega}CH_0(X_v)\ar[d]^{\prod_v\phi_{m,v}}\ar[r]^{\gamma}&Hom(Br(X),\mathbb{Q}/\mathbb{Z})\ar[d]
\\H^{2d}(X,\Z/m\Z(d))\ar[r]^-{\psi_m^H}&\prod'_{v\in\Omega}\tilde{H}^{2d}(X_v,\Z/m\Z(d))\ar[r]&H^2(X,\mathbb{Z}/m\mathbb{Z}(1))^*
}$$
Ici, le groupe $\prod'_{v\in\Omega}\tilde{H}^{2d}(X_v,\Z/m\Z(d))$ est défini comme le produit restreint des
$\tilde{H}^{2d}(X_v,\Z/m\Z(d))$ par rapport aux sous-groupes
$H^{2d}(\mathcal{X}_v,\Z/m\Z(d))\subset H^{2d}(X_v,\Z/m\Z(d)),$ où $\mathcal{X}$ est un modèle entier de $X$ au-dessus d'un
ouvert non vide de $Spec(O_k).$ Le groupe
$H^2(X,\mathbb{Z}/m\mathbb{Z}(1))^*=Hom(H^2(X,\mathbb{Z}/m\mathbb{Z}(1)),\mathbb{Q}/\mathbb{Z})$ est
le dual de Pontryagin de $H^2(X,\mathbb{Z}/m\mathbb{Z}(1)).$

\begin{defn}\label{notionAF}Soit $S\subset\Omega_k$ un ensemble (pas forcément fini) de places de $k.$\ \\
\indent($\star$)On dit que $X$ satisfait \emph{l'approximation au niveau de la cohomologie en $S$} pour les zéro-cycles,
si pour tout $m\in\mathbb{Z}_{>0}$ et pour toute famille $\{z_v\}\in\prod_{v\in\Omega}CH_0(X_v),$
il existe une classe de zéro-cycle global $z=z_m\in CH_0(X)$ tel que $z$ et $\{z_v\}_{v\in S}$ ont même image
dans $\prod_{v\in S}\tilde{H}^{2d}(X_v,\Z/m\Z(d)).$

($\star\star$)On dit que \emph{l'obstruction de Brauer-Manin est la seule à l'approximation au niveau de la cohomologie en $S$},
si pour tout $m\in\mathbb{Z}_{>0}$ et pour toute famille $\{z_v\}\in\prod_{v\in\Omega}CH_0(X_v)$ orthogonale à $Br(X),$
il existe une classe de zéro-cycle global $z=z_m\in CH_0(X)$  tel que $z$ et $\{z_v\}_{v\in S}$ ont même image
dans $\prod_{v\in S}\tilde{H}^{2d}(X_v,\Z/m\Z(d)).$

\medskip
- On dit \emph{l'approximation forte} (resp. \emph{l'approximation forte aux places finies})
si ($\star$) vaut pour  $S=\Omega$ (resp. $S=\Omega^\textmd{f}$).

- On dit \emph{l'approximation faible} (resp. \emph{l'approximation faible aux places finies}) si ($\star$)
vaut pour tout sous-ensemble fini $S\subset\Omega$ (resp. $S\subset\Omega^\textmd{f}$).

- De manière similaire, on définit \emph{l'obstruction de Brauer-Manin est la seule à l'approximation faible/forte}
(resp. \emph{aux places finies}) en utilisant ($\star\star$).
\end{defn}

\begin{rem}\label{remark-definition AF}\ \\
\indent(i)On peut également définir l'obstruction associée à un sous-groupe $B\subseteq Br(X)$
en remplaçant l'application $\gamma$ par la composition de $\gamma$ avec la surjection
$$Hom(Br(X),\mathbb{Q}/\mathbb{Z})\twoheadrightarrow Hom(B,\mathbb{Q}/\mathbb{Z}).$$
Ici, pour une fibration $\pi:X\to C,$
on considère souvent l'obstruction de Brauer-Manin associée au sous-groupe $\pi^*Br(C)\subseteq Br(X).$

(ii)On peut restreindre le diagramme à sous-ensemble $CH_0(X)^\delta=deg^{-1}(\delta)\subset CH_0(X)$ pour définir la
notion d'approximation pour \emph{les zéro-cycles de degré $\delta$}, où
$deg:CH_0(X)\to\mathbb{Z}$ est l'application degré.

En remplaçant les groupes $\tilde{H}^{2d}(X_v,\mathbb{Z}/m\mathbb{Z}(d))$
(resp. $H^{2d}(X,\mathbb{Z}/m\mathbb{Z}(d))$) par $CH_0(X_v)/m$ (resp. $CH_0(X)/m$),
on peut aussi définir la notion d'approximation faible/forte \emph{au niveau du groupe de Chow}.

(iii)Soit $m$ un entier strictement positif fixé,
si $z'_v$ est suffisamment proche de $z_v$ pour $v\in S\subset\Omega^\textmd{f},$ pour tout élément $b$ du groupe
fini\footnote{La finitude de ${_mBr(X_v)}$ se déduit de la suite de Kummer sur $X_v$ plus le théorème de finitude
de la cohomologie étale et la suite spectrale d'Hochschild-Serre.} ${_mBr(X_v)},$  on
a $ \langle z'_v,b \rangle _v= \langle z_v,b \rangle _v\in {_mBr(k_v)}$ pour toute $v\in S\subset\Omega^\textmd{f},$
leurs images dans $\prod_{v\in S}\tilde{H}^{2d}(X_v,\Z/m\Z(d))$
coïncident car le morphisme $$\phi_{m,v}:CH_0(X_v)\To \tilde{H}^{2d}(X_v,\mathbb{Z}/m\mathbb{Z}(d))$$
se factorise comme
$$CH_0(X_v)/m\To ({_mBr(X_v)})^*\subset H^2(X_v,\mathbb{Z}/m\mathbb{Z}(1))^*={H}^{2d}(X_v,\mathbb{Z}/m\mathbb{Z}(d))$$
pour toute $v$ non archimédienne (\textit{cf.} \S 2 de \cite{CT99HP0-cyc} et 2.1(3) de
\cite{Saito0})(on ne peut pas dire grand-chose sur les places réelles).

(iv)Il n'y a pas de raison que l'approximation faible pour les points rationnels au sens classique implique
l'approximation faible au niveau de la cohomologie pour les zéro-cycles (de degré $1$).
\end{rem}


\subsection{Cas où $X=C$ est une courbe}

Les résultats pour $X=C$ une courbe viennent principalement des articles de S. Saito \cite{Saito} et de Colliot-Thélène
\cite{CT99HP0-cyc}. Dans son article \cite{Saito}, S. Saito a été le premier à montrer que,
en supposant la finitude de $\sha(Jac(C)),$
l'obstruction de
Brauer-Manin est la seule au principe de Hasse pour les zéro-cycles de degré $1$. Dans \cite{Eriksson},
Eriksson et Scharaschkin ont donné un résultat plus précis et une démonstration plus simple du même résultat.
Dans \cite{CT99HP0-cyc}, Colliot-Thélène a discuté l'approximation pour les zéro-cycles, en particulier,
il a redémontré le résultat de Saito.
Avec un argument supplémentaire dû à Colliot-Thélène, on a le résultat plus utile
et plus précis suivant,
qui implique que l'obstruction de Brauer-Manin est la seule à l'approximation forte au niveau du groupe de Chow
pour les zéro-cycles de degré $1.$

\begin{thm}[Colliot-Thélène]\label{CTresult}
Soit $C$ une courbe projective lisse géométriquement intègre sur un corps de nombres $k$ avec $\sha(Jac(C))$ fini.
On suppose qu'il existe une famille de zéro-cycles $\{z_v\}_{v\in\Omega}$ de degré $\delta$ telle que
$\{z_v\}\bot Br(C).$

On fait l'hypothèse qu'il existe un zéro-cycle de degré $1$ sur $C.$

Alors, pour tout entier strictement positif $m,$ il existe un zéro-cycle $z_m\in Z_0(C)$ de degré $\delta$
et pour chaque $v\in\Omega$ un zéro-cycle $t_v$ de degré $0,$ tels que
$z_m\sim z_v+mt_v\in Z_0(C_v)$ pour toute $v\in\Omega.$
\end{thm}

\begin{proof}
On fixe un zéro-cycle de degré un $z_0\in Z_0(C).$
On trouve une immersion fermée $C\To J$ définie par $P\mapsto P-z_0,$ où $J=Jac(C)$ est la jacobienne de $C.$
On identifie le groupe des classes de zéro-cycles de degré zéro $CH_0(C)^0$ avec $J(k)$
(resp. $CH_0(C_v)^0$ avec $J(k_v)$ pour toute $v\in\Omega$).

D'après le théorème 0.2 (appliqué à $id:C\to C$) de van Hamel \cite{vanHamel} (et la remarque 1.1(ii) de \cite{Wittenberg}),
on trouve le diagramme suivant dont la première ligne est une suite exacte.
\SelectTips{eu}{12}$$\xymatrix@C=20pt @R=10pt{
\widehat{J(k)}\ar[r]\ar[d] & \prod_{v\in\Omega}J(k_v)_\bullet\ar[r]\ar[d] & Hom(Br(C),\mathbb{Q}/\mathbb{Z})\\
\widehat{J(k)}/m\widehat{J(k)}\ar[r] & \prod_{v\in\Omega}J(k_v)_\bullet/mJ(k_v)_\bullet &
}$$
Ici $\widehat{\mbox{       }}$ est le complété profini, d'après le théorème de Mordell-Weil
$J(k)/mJ(k)\simeq \widehat{J(k)}/m\widehat{J(k)}.$ Le groupe $J(k_v)_\bullet$ est le groupe compact $J(k_v)$ si $v$ est une place
non archimédienne, il est le groupe des composantes connexes $\pi_0(J(k_v))$ si $v$ est une place archimédienne.

Maintenant, soit  $\{z_v\}$ une famille de zéro-cycles de degré $\delta$ orthogonale à $Br(C),$
on note $z'=\delta z_0$ un zéro-cycle global  de degré $\delta.$
On considère $\{z'-z_v\}\in \prod_{v\in\Omega}J(k_v)_\bullet,$ il donne $0$ dans $Hom(Br(C),\mathbb{Q}/\mathbb{Z})$
car $\{z'-z_v\}\bot Br(C).$ Avec le diagramme ci-dessus, pour chaque $m$ entier strictement positif, on trouve
une classe de zéro-cycles $cl(z'_m)\in CH_0(C)^0=J(k)$
telle que les images de $z'_m$ et de $z'-z_v$ soient les mêmes dans $J(k_v)_\bullet/mJ(k_v)_\bullet$ pour toute $v\in\Omega.$
Si $v$ est une place complexe, $J(k_v)_\bullet/mJ(k_v)_\bullet=J(k_v)/mJ(k_v)=0$ car $J(\mathbb{C})$ est $m$-divisible.
Si $v$ est une place réelle, $J(k_v)_\bullet/mJ(k_v)_\bullet=(J(k_v)/2)/m=J(k_v)/mJ(k_v)$ car
$J(\mathbb{R})^o=N_{\mathbb{C}/\mathbb{R}}(J(\mathbb{C}))=2J(\mathbb{R})$ est $m$-divisible.
Si $v$ est une place non archimédienne, c'est clair que $J(k_v)_\bullet/mJ(k_v)_\bullet=J(k_v)/mJ(k_v).$
Donc $z_v$ et $z_m=z'-z'_m$ ont même image dans $J(k_v)/mJ(k_v)$ pour toute $v\in\Omega,$
ceci démontre l'énoncé voulu. On remarque que cet argument ne marche que pour une courbe.
\end{proof}

\begin{rem}\label{condition on CTresult}
(i)D'après \cite{CT99HP0-cyc}, Proposition 3.3, l'hypothèse qu'il existe un zéro-cycle
de degré $1$ sur $C$ est automatiquement satisfaite
lorsque'il existe une famille de zéro-cycles locaux de degré $1$ orthogonale à $Br(C)$
(en supposant la finitude du groupe $\sha(Jac(C))$).
C'est le cas si $C=\mathbb{P}^1$ ou si $\delta=1.$

(ii)Sans supposer qu'il existe un zéro-cycle de degré $1$ sur $C,$ la même conclusion du théorème \ref{CTresult}
reste valable quand
$k$ est totalement imaginaire.
En fait, on peut faire l'argument suivant.
On fixe un zéro-cycle global $z_0$ de $C$ et on note $\delta_0$ son degré. Étant donnés
$\{z_v\}\bot Br(C)$ et $m\in\mathbb{Z}_{>0}$ comme dans le théorème \ref{CTresult}, vu l'injectivité
$CH_0(C_v)/m\hookrightarrow H^2(C_v,\mathbb{Z}/m\mathbb{Z}(1))=\tilde{H}^2(C_v,\mathbb{Z}/m\mathbb{Z}(1))$
pour toute $v\in\Omega^\textmd{f},$
la proposition 3.3 de \cite{CT99HP0-cyc} (appliquée au $m\delta_0$) donne un zéro-cycle global $z'_{m\delta_0}$
et un zéro-cycle local $x'_v$
pour chaque $v\in\Omega^\textmd{f},$ tels que $z'_{m\delta_0}\sim z_v+m\delta_0 x'_v$ pour toute $v\in\Omega^\textmd{f}.$
On note $\delta'$ le degré de $x'_v$ (qui ne dépend pas de $v\in\Omega^\textmd{f}$), on vérifie que
$z_m=z'_{m\delta_0}-m\delta'z_0$ et $t_v=\delta_0 x'_v-m\delta'z_0$ pour toute $v\in\Omega^\textmd{f}$ satisfont la conclusion
du théorème \ref{CTresult}. Pour une place complexe $v,$ le zéro-cycle $z_m-z_v$ est de degré $0,$ puisque
$CH_0(C_v)^0=J(\mathbb{C})$ est $m$-divisible,
où $CH_0(C_v)^0=ker(deg:CH_0(C_v)\To \mathbb{Z})$ et où $J$ est la jacobienne de $C,$
le zéro-cycle $z_m-z_v$ s'écrit alors comme $mt_v$ pour un certain $t_v\in Z_0(C_v).$

(iii)Soit $\{z_v\}_{v\in\Omega}$ une famille de zéro-cycles locaux sur $C,$
l'hypothèse $\{z_v\}\bot Br(C)$ implique que $deg(z_v)=\delta_v=\delta^\textmd{f}$ ne dépend pas de
$v\in\Omega^\textmd{f}.$
En fait,
soient $v_1,v_2\in\Omega^\textmd{f}$ deux places non archimédiennes différentes,
on considère $\{\alpha_v\}\in\bigoplus_{v\in\Omega} Br(k_v)$ tel que $inv_{v_1}(\alpha_{v_1})=1/n,$
$inv_{v_2}(\alpha_{v_2})=-1/n,$ et $inv_{v}(\alpha_{v})=0$ pour toutes les autres places $v.$
D'après la suite exacte $0\to Br(k)\to \bigoplus_{v\in\Omega} Br(k_v)\to \Q/\Z\to0,$ la famille
$\{\alpha_v\}_{v\in\Omega}$ provient d'élément $\alpha\in Br(k),$ on note encore $\alpha$ son image dans $Br(C).$
La condition $\langle\{z_v\},\alpha\rangle=0$ implique que $\delta_{v_1}\equiv\delta_{v_2}(mod\mbox{ }n).$
En variant $n\in\mathbb{Z}_{>0},$ on trouve $\delta_{v_1}=\delta_{v_2}.$
De la même manière, on a
$deg(z_v)\equiv\delta^\textmd{f}(mod\mbox{ }2)$ pour toute place réelle $v.$ Cet argument reste valable pour
une variété quelconque au lieu d'une courbe $C.$
\end{rem}


\section{Énoncés des théorèmes principaux et quelques applications}
Premièrement, on énonce les théorèmes plus précis \ref{thm4}, \ref{thm5}. Ensuite, on donne quelques applications.


\subsection{Théorèmes principaux}\ \\

\textbf{Principe de Hasse pour les zéro-cycles sur $X$}
\begin{thm}\label{thm4}
Soit $\pi:X\To C$ une fibration à fibres géométriquement intègres avec $\sha(Jac(C))$  fini.
Soient $\textsf{Hil}$ un sous-ensemble hilbertien généralisé de $C$ et $\delta$ un entier.
Supposons que pour tout point fermé $\theta\in \textsf{Hil},$
la fibre $X_\theta$ satisfait le principe de Hasse (pour les zéro-cycles de degré $1$ ou pour les points rationnels).

Alors, l'obstruction de Brauer-Manin associée au sous-groupe $\pi^*Br(C)\subset Br(X)$ est la seule
à l'existence d'un zéro-cycle de degré $1$ sur $X.$
\end{thm}

\textbf{Approximation faible/forte pour les zéro-cycles sur $X$}

\begin{thm}\label{thm5}
Soit $\pi:X\To C$ une fibration à fibres géométriquement intègres avec $\sha(Jac(C))$ fini.
Soient $\textsf{Hil}$ un sous-ensemble hilbertien généralisé de $C$ et $\delta$ un entier.

On suppose, lorsque $k$ a une place réelle,
qu'il existe une famille de zéro-cycles locaux de degré $1$ orthogonale à $\pi^*Br(C).$

(1)On suppose que pour
tout point fermé $\theta\in \textsf{Hil},$ la fibre $X_\theta$ satisfait l'approximation faible pour les points rationnels
(resp. l'approximation faible au niveau de la cohomologie aux places finies pour les zéro-cycles de degré $1$).

Alors, l'obstruction de Brauer-Manin associée au sous-groupe $\pi^*Br(C)\subset Br(X)$
est la seule à l'approximation faible au niveau
de la cohomologie aux places finies pour les zéro-cycles de degré $\delta$.

(2)On suppose, de plus, que\ \\
(H CH0) l'application $CH_0(X_v)\To CH_0(C_v)$ est injective pour presque toute place $v\in\Omega_k.$

Alors, l'obstruction de Brauer-Manin associée au sous-groupe $\pi^*Br(C)\subset Br(X)$
est la seule à l'approximation forte au niveau
de la cohomologie aux places finies pour les zéro-cycles de degré $\delta$.
\end{thm}

\begin{rem}\label{remark-CH0}
L'hypothèse (H CH0) que $CH_0(X_v)\To CH_0(C_v)$ soit injective pour presque toute place
$v\in\Omega_k$ est vérifiée pour
certaines fibrations. Par exemple, si la fibre générique $X_\eta$ est une quadrique sur $k(C)$ de dimension $d\geqslant 3,$
la fibration $X\To C$ vérifie (H CH0), voir \cite{ParimalaSuresh}, Théorème 5.2, et les références de cet article-là.
Si la fibre générique $X_\eta$ est une variété de Severi-Brauer d'indice sans facteurs carrés,
la fibration $X\To C$ vérifie également (H CH0), voir \cite{Fro98}, Théorème 4.8. Plus généralement,
dans son article \cite{Wittenberg}, Wittenberg démontre que
l'hypothèse (H CH0) est vérifiée si la fibre générique $X_\eta$ est rationnellement connexe sur $\overline{k(C)},$
\textit{cf.} le corollaire 2.2 de \cite{Wittenberg}.
\end{rem}

\begin{rem}
En concernant le théorème \ref{thm4}, la démonstration ci-dessous utilise un argument de type d'approximation
pour les zéro-cycles sur $C$ (Théorème \ref{CTresult}, Lemme \ref{lem3}). Donc, même si l'on suppose qu'il existe
un zéro-cycle global de degré $1$ sur $C,$ l'hypothèse sur la finitude de $\sha(Jac(C))$ est inévitable.
Cependant, Colliot-Thélène montre dans \cite{CTsurPoonen} ce théorème pour les solides de Poonen (où $C(k)\neq\emptyset$
\emph{cf.} les applications ci-dessous)
sans supposer la finitude de $\sha(Jac(C)).$
\end{rem}

\begin{cor}\label{corofsection2}
Avec les même notations des théorèmes \ref{thm4} et \ref{thm5}, soit $C=\mathbb{P}^1$ la droite projective.
Supposons la même hypothèse que celle du théorème \ref{thm4} (resp. du théorème \ref{thm5})
sur les fibres.
Alors le principe de Hasse (resp. l'approximation faible/forte au niveau de la cohomologie aux places finies) vaut
pour les zéro-cycles de degré $1$ (resp. de degré $\delta$) sur $X.$
\end{cor}

La partie concernant le principe de Hasse pour les zéro-cycles de degré $1$ dans ce corollaire (avec $\textsf{Hil}$
l'ensemble des points fermés d'un certain ouvert non vide de $C$) est également une
conséquence du théorème 4.1 de Colliot-Thélène/Skorobogatov/ Swinnerton-Dyer  \cite{CT-Sk-SD}.


\subsection{Applications}\ \\

\textbf{Surfaces réglées et fibrations en variétés de Severi-Brauer}

Les résultats de ce paragraphe ont été montrés par Colliot-Thélène \cite{CT99}, Frossard \cite{Frossard},
et van Hamel \cite{vanHamel}.

Soit $X\To C$ une surface réglée  (\textit{i.e.} une surface fibrée en coniques: sa fibre générique est un conique)
à fibres géométriquement intègres,
en supposant la finitude du groupe $\sha(Jac(C)),$
d'après les théorèmes \ref{thm4}, \ref{thm5},
l'obstruction de Brauer-Manin
associée à $\pi^*Br(C)$ est la seule au principe de Hasse et à l'approximation forte au niveau de
la cohomologie aux places finies pour les zéro-cycles de degré $1.$

Plus généralement, soit $X\To C$ une fibration
en variétés de Severi-Brauer (resp. d'indice sans facteur carré) à fibres géométriquement intègres,
en supposant la finitude du groupe $\sha(Jac(C)),$ d'après les théorèmes \ref{thm4}, \ref{thm5}, l'obstruction de Brauer-Manin
associée à $\pi^*Br(C)$ est la seule au principe de Hasse et à l'approximation faible (resp. forte)
au niveau de la cohomologie aux places finies pour les zéro-cycles de degré $1$ sur $X.$
En fait, lorsqu'une variété de Severi-Brauer a une famille de zéro-cycles de degré $1$ locaux, elle a un point adélique
(par l'argument de restriction-corestriction du groupe de Brauer), ainsi un point global (par
la suite exacte $0\to Brk\to\bigoplus_{v\in\Omega}Brk_v\to\Q/\Z\to 0$),
elle est alors isomorphe à un espace projectif qui satisfait l'approximation faible
au niveau de la cohomologie pour les zéro-cycles de degré $1.$

On remarque que, dans Frossard \cite{Frossard}, on ne demande pas qu'une fibration en variétés de Severi-Brauer soit à fibres
géométriquement intègres, dans ce cas-là, l'obstruction de Brauer-Manin associée au groupe $Br(X)$ est nécessaire,
le sous-groupe $\pi^*Br(C)$ ne suffit pas.
Dans son article récent \cite{Wittenberg}, Wittenberg montre que l'obstruction de Brauer-Manin est la seule au principe de
Hasse (resp. à l'approximation forte au niveau du groupe de Chow) pour les zéro-cycles de degré $1$
sur \emph{toute} fibration en variétés de Severi-Brauer au-dessus une courbe (en supposant la finitude du groupe de
Tate-Shafarevich).

\smallskip
\textbf{Fibrations en surfaces de Châtelet - Solides de Poonen}

Dans un article récent \cite{Poonen}, Poonen
construit une fibration en surfaces de Châtelet $X$ au-dessus d'une courbe $C,$
telle que \ \\
(1) $X$ n'a pas de point $k$-rationnel et \ \\
(2) il n'y a pas d'obstruction de Brauer-Manin\footnote{De plus, il n'y a pas d'obstruction de Brauer-Manin étale,
voir \cite{Poonen} pour plus des détails.} au principe de Hasse pour les points rationnels.

On rappelle que
la fibration $X\To C$ est le pull-back de $V\To\mathbb{P}^1$ via un morphisme dominant $C\To\mathbb{P}^1,$
où $V$ est la compactification standard de son ouvert $V_0$ défini par l'équation
$$y^2-az^2=u^2\tilde{P}_\infty(r,w)+v^2\tilde{P}_0(r,w)$$
dans $\mathbb{P}^1\times\mathbb{P}^1\times\mathbb{A}^1\times\mathbb{A}^1,$ et où la fibration est donnée par la première
projection
${(u:v;r:w;y,z)\mapsto (u:v)}.$
Ici $a\in k^*\setminus k^{*2},$ $\tilde{P}_\infty(r,w)$ et
$\tilde{P}_0(r,w)$ sont les homogénéisations des polynômes $P_\infty(x), P_0(x)\in k[x]$  de degré $4.$
On a une factorisation
\SelectTips{eu}{12}$$\xymatrix@C=20pt @R=14pt{
V\ar[d]\ar[dr] & &(u:v;r:w;y,z)\ar@{|->}[d]\ar@{|->}[dr] &\\
\mathbb{P}^1&\mathbb{P}^1\times\mathbb{P}^1\ar[l] &(u:v)&(u:v;r:w)\ar@{|->}[l]
}$$
Une courbe lisse intègre $Z_1\subset\mathbb{P}^1\times\mathbb{P}^1$ est définie par
$$0=u^2\tilde{P}_\infty(r,w)+v^2\tilde{P}_0(r,w),$$
$Z_1\To \mathbb{P}^1$ est alors un morphisme fini plat.
On restreint à l'ouvert où $v\neq 0,w\neq 0,$ l'équation de la fibration $V\To\mathbb{P}^1$ s'écrit (avec $t=u/v,$ $x=r/w$)
$$V:y^2-az^2=P_t(x);V\To \mathbb{P}^1:(t,x,y,z)\mapsto t,$$ où $P_t(x)=t^2P_\infty(x)+P_0(x)\in k(t)[x].$
(En fait, l'argument suivant marchera bien pour un polynôme $P_t(x)$ quelconque de degré $4$ en $x,$ lorsque
$Z_1\subset\mathbb{P}^1\times\mathbb{P}^1$ défini
par $P_t(x)=0$ est une courbe intègre.)
Pour un morphisme dominant quelconque  $\psi:C\To\mathbb{P}^1,$ le produit
fibré $Z={C\times_{\mathbb{P}^1}Z_1}$ est une courbe intègre
si les lieux de branchement de $Z_1\To\mathbb{P}^1$ et de $C\To\mathbb{P}^1$
ne se rencontrent pas (\textit{cf.} Lemme 7.1 de \cite{Poonen}).
Le revêtement $Z\To C$ définit alors un sous-ensemble hilbertien généralisé $\textsf{Hil}_C$ de $C.$
Pour tout point fermé $\theta$ de $C,$ la fibre $X_\theta$ est définie par $y^2-az^2=P_\theta(x)$ avec
$P_\theta(x)=\psi(\theta)^2P_\infty(x)+P_0(x)\in k(\theta)[x].$ La condition $\theta\in \textsf{Hil}_C$ signifie que
le polynôme $P_\theta(x)$ est irréductible sur $k(\theta),$ la fibre
$X_\theta$ satisfait donc le principe de Hasse et l'approximation faible pour les points rationnels, \textit{cf.}
Théorème 8.11 de Colliot-Thélène/Sansuc/Swinnerton-Dyer \cite{chateletsurfaces}.
On sait aussi que toutes ses fibres sont géométriquement intègres.

Dans \cite{CTsurPoonen}, Colliot-Thélène
a montré qu'il existe un zéro-cycle global de degré $1$ sur une telle $X$ avec quelques hypothèses supplémentaires,
disons

(i) la fibre $V_\infty$ est lisse possédant les points rationnels localement partout,

(ii) l'ensemble $C(k)$ est non vide fini,

(iii) l'image de $C(k)$ par $C\To\mathbb{P}^1$ est contenue dans  $\{\infty\}.$\ \\
(Le point (iii) avec la finitude de $C(k)$ assure que
l'obstruction au principe de Hasse associée à $\pi^*Br(C)$ disparaît, \textit{cf.} \cite{Poonen}.)

En fait, même si l'on n'a plus ces hypothèses supplémentaires, en supposant la finitude de $\sha(Jac(C)),$
le théorème \ref{thm4} (resp. \ref{thm5})
assure que l'obstruction de Brauer-Manin associée au sous-groupe $\pi^*Br(C)\subset Br(X)$ est la seule au
principe de Hasse (resp. à l'approximation forte
au niveau de la cohomologie\footnote{Ceci devrait également valoir au niveau du groupe de Chow, si l'énoncé (E) dans
la remarque à la fin de cet article était établi.} aux places finies)
pour les zéro-cycles de degré $1$ sur $X.$ Ceci redémontre le résultat principal de Colliot-Thélène \cite{CTsurPoonen}
en supposant la finitude de $\sha(Jac(C)).$

\smallskip
\textbf{Autres exemples}

(1)On considère une fibration $X\To \mathbb{P}^1$ définie comme suit.
Sa restriction à $\mathbb{P}^1\setminus\infty$ est définie
dans $\mathbb{P}^2\times\mathbb{A}^1$ par
$$X_0:A(t)x^2+B(t)y^2+C(t)z^2=0$$
$$X_0\To \mathbb{A}^1=\mathbb{P}^1\setminus\infty;(x:y:z,t)\mapsto t$$
où $A(t),B(t),C(t)\in k[t]$ sont des polynômes sans racines communes, on suppose que
$deg(A(t))\equiv deg(B(t))\equiv deg(C(t))(mod\mbox{ }2).$
On pose $$A(t)=a_0+a_1t^1+\cdots+a_{d_1}t^{d_1},$$
$$B(t)=b_0+b_1t^1+\cdots+b_{d_2}t^{d_2},$$
$$C(t)=c_0+c_1t^1+\cdots+c_{d_3}t^{d_3}.$$ On peut supposer que $d_1\geqslant d_2\geqslant d_3.$
La fibration restreinte à $\mathbb{P}^1\setminus 0$
est définie par $$X_\infty:A_\infty(t')x'^2+B_\infty(t')y'^2+C_\infty(t')z'^2=0$$
$$X_\infty\To\mathbb{A}^1=\mathbb{P}^1\setminus 0;(x':y':z',t')\mapsto t',$$ où
$$A_\infty(t')=a_{d_1}+a_{d_1-1}t'^1+\cdots+a_0t'^{d_1},$$
$$B_\infty(t')=b_{d_2}+b_{d_2-1}t'^1+\cdots+b_0t'^{d_2},$$
$$C_\infty(t')=c_{d_3}+c_{d_3-1}t'^1+\cdots+c_0t'^{d_3}.$$
Si l'on recolle $X_0$ et $X_\infty$
via l'isomorphisme donné par $t'=1/t,$ $x'=x,$ $y'=yt^{\frac{d_2-d_1}{2}},$ et $z'=zt^{\frac{d_3-d_1}{2}},$
on définit alors un morphisme $X\To\mathbb{P}^1.$ C'est une fibration (à fibre générique géométriquement intègre)
si le polynôme
$A(t)x^2+B(t)y^2+C(t)z^2\in k(t)[x,y,z]$ est irréductible sur $\overline{k(t)},$ c'est toujours le cas si
$A(t),B(t),C(t)\in k(t)$ sont non nuls.
On peut vérifier que toutes les fibres contiennent une composante irréductible de multiplicité un qui est
géométriquement intègre
si les conditions suivantes sont vérifiées
(les théorèmes concernés sont encore valables avec l'hypothèse sur les fibres,
\textit{cf.} les lemmes 2.2 et 2.3 de Skorobogatov \cite{Sk2}),

(i) $-1\in k^{*2};$

(ii-a) $B(\alpha)/C(\alpha)\in k(\alpha)^{*2},$ pour tout $\alpha\in\bar{k}$ tel que $A(\alpha)=0;$

(ii-b) $C(\alpha)/A(\alpha)\in k(\alpha)^{*2},$ pour tout $\alpha\in\bar{k}$ tel que $B(\alpha)=0;$

(ii-c) $A(\alpha)/B(\alpha)\in k(\alpha)^{*2},$ pour tout $\alpha\in\bar{k}$ tel que $C(\alpha)=0.$\ \\
D'après le corollaire \ref{corofsection2}, le principe de Hasse/l'approximation faible au niveau de
la cohomologie aux places finies vaut pour les zéro-cycles
de degré $\delta$ sur $X.$

Soit $C\To \mathbb{P}^1$ un morphisme non constant quelconque, si l'on suppose que $\sha(Jac(C))$ est fini,
l'obstruction de Brauer-Manin associée à $\pi^*Br(C)$ est la seule au principe de Hasse/à
l'approximation faible au niveau de la cohomologie aux places finies pour les zéro-cycles
de degré $1$ sur $X\times_{\mathbb{P}^1}C$ d'après le théorème \ref{thm4}, \ref{thm5}.

Pour un exemple explicite, on pose $k=\mathbb{Q}(\sqrt{-1},\sqrt{5},\sqrt{6},\sqrt{7}),$ $A(t)=t+3,$ $B(t)=(t^2-2)(t+2),$
$C(t)=t-3,$ et on suppose que $C$ est une courbe hyper-elliptique,
dont le groupe $\sha(Jac(C))$ est fini, définie par
$u^2=P(v)$ avec $P(v)\in k[v].$
Alors, le principe de Hasse/l'approximation faible au niveau de la cohomologie aux places finies
vaut pour les zéro-cycles sur la variété projective définie par l'équation affine
$$(t+3)x^2+(t^2-2)(t+2)y^2+(t-3)=0.$$ L'obstruction de Brauer-Manin (associée à $\pi^*Br(C)$) au principe de Hasse/à
l'approximation faible au niveau de la cohomologie aux places finies
pour les zéro-cycles de degré
$1$ est la seule pour la variété projective définie sur $k$ par l'équation affine
$$(\sqrt{P(v)}+3)x^2+(P(v)-2)(\sqrt{P(v)}+2)y^2+(\sqrt{P(v)}-3)=0,$$ si $\sha(Jac(u^2=P(v)))$ est un groupe fini.

(2)Pour une fibration au-dessus de $\mathbb{P}^1$ dont la fibre générique est définie par l'équation
$$\sum_{i=1}^na_i(t)x_i^2=0$$ avec $a_i(t)\in k(t)$ de degré pair (ou impair)
tels que les $div_{\mathbb{P}^1}(a_i)$ sont différents deux à deux et $n\geqslant 4,$
de la même façon que (1), on trouve facilement un modèle projectif explicite $X\To\mathbb{P}^1,$ qui est une fibration
à fibre générique géométriquement intègre si au moins trois des $a_i(t)$ sont non nuls.
On vérifie que toutes les fibres contiennent une composante irréductible de multiplicité un qui est
géométriquement intègre.
Pour un morphisme non constant entre des courbes $C\To\mathbb{P}^1$ avec $\sha(Jac(C))$ un groupe fini, on pose
$X_C=X\times_{\mathbb{P}^1}C.$ D'après le théorème \ref{thm4}, le principe de Hasse vaut pour les zéro-cycles de degré $\delta$
sur $X,$ l'obstruction de Brauer-Manin au principe de Hasse pour les zéro-cycles de degré
$1$ sur $X_C$ est la seule.


\section{Preuves des théorèmes}

On va démontrer les théorèmes \ref{thm4} et \ref{thm5}. Un outil important consiste en les lemmes de déplacement,
qui permettent de ramener les questions sur les zéro-cycles aux questions sur les points rationnels.
On applique la méthode de Colliot-Thélène \cite{CT99}, voir aussi \cite{Frossard},
avec le théorème des fonctions implicites et le théorème de Lang-Weil, on
obtient l'existence d'un zéro-cycle global.
À l'aide de l'application de Gysin en cohomologie étale, on arrive aux énoncés autour de l'approximation pour
les zéro-cycles.


\subsection{Lemmes de déplacement}

Les lemmes de ce type sont connus depuis longtemps, on les énonce ici pour le confort du lecteur.

\begin{lem}\label{deplacement}
Soient $X$ une variété  intègre régulière  sur un corps parfait infini $k,$ et $U$ un ouvert non vide de $X.$
Alors tout  zéro-cycle $z$ de $X$ est rationnellement équivalent, sur $X,$ à un zéro-cycle $z'$ à support dans $U.$
\end{lem}

\begin{proof}
On trouve une démonstration en détails dans \cite{CT05} \S 3.
\end{proof}

\begin{lem}\label{deplacement1}
Soit $\pi:X\To C$ un morphisme propre non constant au-dessus d'une courbe $C$ sur un corps parfait infini $k.$
On suppose que $X$ et $C$ sont des variétés lisses projectives géométriquement intègres sur $k.$
Soit $D$ un ensemble fini de points fermés de $C,$
on fixe $z_0$ un zéro-cycle effectif de $X.$

Alors, pour tout $z\in Z_0(X),$ il existe un entier positif $r_0$ tel que, pour tout $r>r_0,$
$z+rz_0$ soit rationnellement équivalent sur $X$ à un zéro-cycle effectif $z_r$ tel que $\pi_*(z_r)$
soit séparable à support en dehors de $D$ et de $\pi_*(z_0).$
\end{lem}

\begin{proof}
Voir \cite{CT99} Lemme 3.2, la même démonstration fonctionne pour $z_0$ remplaçant un point fermé.
\end{proof}

\subsection{Lang-Weil $+$ Hensel}
Quand on applique la méthode des fibrations, on a besoin d'estimations de Lang-Weil avec le lemme de Hensel
pour assurer que des fibres ont des points rationnels locaux pour presque toutes les places.

\begin{lem}\label{Lang-Weil}
Soit $X\To C$ une fibration à fibres géométriquement intègres,
il existe alors un sous-ensemble
fini $T$ de places de $k$ tel que pour tout point fermé $\theta\in C$ avec $X_\theta$
lisse on ait $X_\theta(k(\theta)_w)\neq\emptyset$ pour toute place
$w$ au-dessus d'une place  $v\in\Omega_k\setminus T.$
\end{lem}
\begin{proof}
La preuve suivante suit une partie de la preuve du lemme 1.2 de Colliot-Thélène/Skorobogatov/Swinnerton-Dyer \cite{CT-Sk-SD}.

Il existe un sous-ensemble fini $T$ de places de $k$ et un modèle entier $\Pi:\mathcal{X}\To\mathcal{C}$
au-dessus de $O_T$ de l'application $\pi:X\To C$ avec $\Pi$ un morphisme projectif plat.
On pose $U=C\setminus D$ et on note $\mathcal{U}$ un modèle de $U$ au-dessus de $O_T,$
où $D$ est le sous-ensemble des points fermés correspondant à une fibre non lisse.
En augmentant $T,$ on peut demander également que toutes les fibres de la restriction $\Pi|_\mathcal{U}$
soient géométriquement intègres lisses.
Le polynôme de Hilbert
est constant pour une famille projective plate d'espace de paramètre connexe.
Toutes les fibres de $\Pi|_\mathcal{U}$ sont alors géométriquement intègres projectives lisses de dimension $d$ fixée
et de degré $n$ (pour la famille entière)  dans un espace projectif de dimension $r$ fixée. La fibre en un point
fermé $u$ de $\mathcal{U}$ est projective lisse géométriquement intègre sur un corps fini $k(u).$ D'après
les estimations de Lang-Weil \cite{Lang-Weil}, si la cardinalité de $k(u)$   est plus grande qu'une constante $c,$ qui ne
dépend que de $d,$ $n$ et $r,$
une telle fibre a un  $k(u)$-point (lisse). Le lemme de Hensel assure que $X_\theta(k(\theta)_w)\neq\emptyset$ ($\theta\in U$)
où $u$ est dans l'adhérence de $\theta$ dans $\mathcal{C}$ définissant une place non archimédienne de $k(\theta).$
En augmentant $T,$ on peut supposer que pour toute $w\in \Omega_{k(\theta)}$ en dehors de $T$ on ait
$|k(\theta)(w)|>c.$
\end{proof}

\subsection{Lemme d'irréductibilité de Hilbert}
Avec le lemme suivant, qui est en un certain sens une version effective du théorème d'irréductibilité de Hilbert,
on peut approximer certains zéro-cycles effectifs locaux par un point fermé global
appartenant à un sous-ensemble hilbertien généralisé $\textsf{Hil}.$
C'est la version pour les zéro-cycles du théorème 1.3 de Ekedahl \cite{Ekedahl}.

\begin{lem}\label{simplified}
Soit $C$ une courbe projective lisse géométriquement intègre de genre $g$
sur un corps de nombres $k.$
Soit $\textsf{Hil}$ un sous-ensemble hilbertien généralisé de $C.$
Soit $y\in Z_0(C)$ un zéro-cycle effectif de degré $d>2g.$
On fixe $S$ un sous-ensemble fini de $\Omega_k.$
On suppose que $z_v\in Z_0(C_v)$ est un zéro-cycle effectif séparable de degré $d$ à support disjoint de $supp(y)$
rationnellement équivalent à $y\times_kk_v$ sur $C_v$ pour tout $v\in S.$

Alors il existe un point fermé $\theta$ de $C$ de degré $d,$ tel que

(1) $\theta\in\textsf{Hil},$

(2) $\theta$ soit rationnellement équivalent à $y,$

(3) $\theta$ soit suffisamment proche de $z_v$ pour tout $v\in S.$
\end{lem}

\begin{proof}
Pour chaque $v\in S,$ on écrit $z_v-y=div_{C_v}(f_v)$ avec une certaine fonction $f_v\in k_v^*(C_v)/k_v^*.$
Comme $deg(y)=d>2g,$
d'après le théorème de Riemann-Roch, $\Gamma(C,\mathcal{O}_C(y))$ est un espace vectoriel de dimension
$r=d+1-g>g+1.$
l'approximation faible appliquée à l'espace projectif $\mathbb{P}^{r-1}$ nous donne une fonction $f\in k(C)^*/k^*$ tel que

(i) $f$ soit suffisamment proche de $f_v$ pour tout $v\in S,$

(ii) $div_C(f)=y'-y$ avec $y'$ un zéro-cycle effectif à support hors de $supp(y).$

D'après une version convenable du lemme de Krasner, (i) implique que, pour tout $v\in S,$
le zéro-cycle $y'$ est suffisamment proche de $z_v$ et donc séparable.

La fonction $f$ définit alors
un $k$-morphisme $\psi:C\To \mathbb{P}^1$ tel que $\psi^*(\infty)=y$ et $\psi^*(0)=y'.$
Supposons que le sous-ensemble hilbertien généralisé $\textsf{Hil}$ est défini par un morphisme fini étale
$Z\To U\subset C$ avec $Z$ une variété intègre.
La composition $\varphi:Z\To C\To\mathbb{P}^1$ définit un sous-ensemble hilbertien généralisé
$\textsf{Hil}'$ de $\mathbb{P}^1$ de la façon suivante:
En enlevant un ensemble fini $\psi(C\setminus U)$ de points fermés et les points ramifiés, on trouve un ouvert non vide
$U'$ de $\mathbb{P}^1$ tel que $\psi^{-1}(U')\subset U$ et tel que $\varphi$ soit étale sur $Z'=\varphi^{-1}(U'),$
le morphisme $\varphi_{|Z'}:Z'\to U'\subset\mathbb{P}^1$ définit alors $\textsf{Hil}'.$
Une fois qu'un point fermé $\theta'$ de $\mathbb{P}^1$ appartient à $\textsf{Hil}',$
l'image réciproque $\theta=\psi^{-1}(\theta')$ est un point fermé de $C$ appartenant à $\textsf{Hil}.$
D'après le théorème d'irréductibilité de Hilbert (de version effective par Ekedahl \cite{Ekedahl}, Théorème 1.3), il existe un
point $k$-rationnel $\theta'$ de $\mathbb{P}^1$ appartenant à $\textsf{Hil}'$ et suffisamment proche de
$0\in\mathbb{P}^1(k_v)$ pour chaque $v\in S.$ Comme $y'=\psi^*(0)$ est un zéro-cycle séparable, le morphisme $\psi$ est alors
étale en chaque point fermé apparaissant dans $y',$ d'où $\theta=\psi^*(\theta')=\psi^{-1}(\theta')\in \textsf{Hil}$
est un zéro-cycle suffisamment proche de $y'\times_kk_v$ pour tout $v\in S.$ Vus comme zéro-cycles sur $C,$ on a $\theta\sim y.$
\end{proof}

\subsection{Proposition clé}
La proposition suivante joue un rôle crucial.

\begin{prop}\label{prop1}
Soient $\pi:X\To C$ une fibration à fibres géométriquement intègres et $D$ un sous-ensemble fini
de points fermés de $C.$ Soit $\textsf{Hil}$ un sous-ensemble hilbertien généralisé de $C.$
Soit $y\in Z_0(C)$ un zéro-cycle (pas nécessairement effectif) à support en dehors de $D.$ Supposons qu'il existe une famille
$\{z_v\}$ de zéro-cycles de $X_v$ telle que $z_v\in Z_0(X_v/y)$ pour toute $v\in \Omega.$ Alors,

(0)Pour tout sous-ensemble fini $S$ de $\Omega,$ il existe un entier positif $s_0$ tel que, pour tout entier
$s>s_0,$
il existe les données suivantes:

(i)un zéro-cycle effectif $z_0\in Z_0(X)$ tel que $y_0=\pi_*(z_0)\in Z_0(C)$ est à support hors de $D;$

(ii)pour chaque $v\in S,$ un zéro-cycle effectif $\tau_v\in Z_0(X_v)$ rationnellement équivalent à $z_v+(s+1)z_0$ sur $X_v$
tel que $\pi_*(\tau_v)$ soit séparable à support
hors de $D;$

(iii)un point fermé  $\theta \in\textsf{Hil}$ de degré $d$ tel que
$\theta$ soit déployé localement partout, et tel que comme zéro-cycle $\theta$ soit suffisamment proche de $\pi_*(\tau_v)$
pour toute $v\in S,$ et $\theta-(s+1)y_0$ soit
rationnellement équivalent à $\pi_*(z_v)$ sur $C_v$ pour toute $v\in\Omega.$

(1)De plus, on suppose que pour tout point fermé $\theta\in \textsf{Hil},$
la fibre $X_\theta$ satisfait le principe de Hasse
pour l'existence de points rationnels (ou l'existence de zéro-cycles de degré $1$). Alors il existe un zéro-cycle
$z\in Z_0(X)$ tel que $\pi_*(z)$ soit suffisamment proche de $\pi_*(\tau_v)-(s+1)y_0\in Z_0(C_v)$ pour toute $v\in S$
et tel que $\pi_*(z)\sim\pi_*(z_v)$ pour toute $v\in\Omega.$

(2)De plus, on suppose que pour tout point fermé $\theta\in \textsf{Hil},$
la fibre $X_\theta$ satisfait l'approximation faible
pour les points rationnels. Alors on peut choisir  $z\in Z_0(X)$ tel que
$z$ soit suffisamment proche de $\tau_v-(s+1)z_0$ et tel que $\tau_v-(s+1)z_0\sim z_v\in Z_0(X_v),$ ainsi
$\pi_*(z)$ soit suffisamment proche de $\pi_*(\tau_v)-(s+1)y_0$
et $\pi_*(\tau_v)-(s+1)y_0\sim \pi_*(z_v)\in Z_0(C_v)$ pour toute $v\in S.$
\end{prop}

\begin{proof}
D'après les estimations de Lang-Weil et le lemme de Hensel (\textit{cf.} Lemme \ref{Lang-Weil}),
il existe un sous-ensemble fini $T\subset \Omega$ tel que pour tout point fermé
$x\in C\setminus D$ la fibre $X_x$ ait un $k(x)_w$-point rationnel pour toute place $w$ au-dessus d'une place
$v\in\Omega_k\setminus T.$

On peut supposer que $S\supseteq T.$

On prend un zéro-cycle effectif $z_0$ qui est un multiple d'un point fermé de $X$ tel que $deg(z_0)$ soit suffisamment grand, et
tel que $y_0=\pi_*(z_0)\in Z_0(C)$ soit à support hors de $D$. Le lemme 3.1 de \cite{CT99} implique qu'il
existe  un zéro-cycle effectif séparable $y_1\sim y+y_0\in Z_0(C)$ à support en dehors de $D.$

On pose $t_v=z_v+z_0\in Z_0(X_v/y_1).$ D'après le lemme de déplacement \ref{deplacement1}, on trouve
alors un entier $s_0$ suffisamment grand (qui dépend de $S$)
tel que pour tout $s>s_0$ et toute $v\in S$ il existe un zéro-cycle effectif
$\tau_v\sim t_v+sz_0=z_v+(s+1)z_0\in Z_0(X_v)$ tel que $\pi_*(\tau_v)$ soit séparable, déployé, et à
support en dehors de $D\cup supp(y_0)\cup supp(y_1).$

On a alors, pour toute $v\in S,$
$\pi_*(\tau_v)\sim y_1+sy_0\in Z_0(C_v),$ $deg(\pi_*(\tau_v))$ suffisamment grand.
On applique le lemme \ref{simplified} et on trouve $\theta\in \textsf{Hil}.$
Pour $v\in S$
l'image de ${\theta}$ dans $Z_0(C_v)$ est déployée par le théorème des fonctions implicites, pour $v\notin S$ c'est aussi le cas
par les estimations de Lang-Weil.
Donc ${\theta}$ est déployé localement partout.

On sait que ${\theta}\sim y_1+sy_0,$ $y_1\sim y+y_0$ et $z_v\in Z_0(X_v/y)$ pour toute $v\in\Omega,$
donc ${\theta}-(s+1)y_0\sim\pi_*(z_v)$
pour toute $v\in\Omega.$
On sait aussi que pour toute $v\in S,$ ${\theta}\sim y_1+sy_0\sim \pi_*(\tau_v)\in Z_0(C_v).$
Ceci complète la preuve de l'assertion (0).

(1)De plus, si pour tout point fermé appartenant à $\textsf{Hil},$  la fibre en ce point satisfait le principe de Hasse,
alors il existe un zéro-cycle global de degré $1$ sur la $k(\theta)$-variété $X_\theta,$
\textit{i.e.} il existe $z'\in Z_0(X)$ tel que $\pi_*(z')={\theta},$
on pose $z=z'-(s+1)z_0$ alors $\pi_*(z)={\theta}-(s+1)y_0.$

(2)De plus, si pour tout point fermé appartenant à $\textsf{Hil},$
la fibre en ce point satisfait l'approximation faible pour les points rationnels,
alors on peut choisir $z'\in Z_0(X)$ suffisamment proche de $\tau_v$ pour $v\in S$ par le
théorème des fonctions implicites, alors $z=z'-(s+1)z_0$ est suffisamment proche de $\tau_v-(s+1)z_0$
et $\tau_v-(s+1)z_0\sim z_v\in Z_0(X_v)$ pour toute $v\in S.$
\end{proof}

\begin{rem}
Dans l'assertion (2), il semble qu'on ne peut pas avoir $z\sim \tau_v-(s+1)z_0\sim z_v\in Z_0(X_v)$
généralement. Dans des cas particuliers d'une fibration en coniques au-dessus d'une
courbe ou une fibration en variétés de Severi-Brauer au-dessus d'une courbe, on peut assurer
dans les assertions respectivement à (2) que le zéro-cycle global obtenu est
localement rationnellement équivalent aux zéro-cycles locaux donnés.
\end{rem}

\subsection{Preuves des théorèmes}

Afin de montrer les théorèmes, on a besoin d'un lemme.

\begin{lem}\label{lem3}
Soit $X\To C$ une fibration avec $\sha(Jac(C))$ fini.
Soit $\{z_v\}\in \prod_{v\in\Omega}Z_0(X_v)$ une famille de zéro-cycles
de degré $\delta$ à support en dehors de $D$ telle que $\{\pi_*(z_v)\}\bot Br(C).$
On suppose, lorsque $k$ a une place réelle,
qu'il existe une famille de zéro-cycles locaux de degré $1$ orthogonale à $\pi^*Br(C).$

Alors pour tout entier strictement positif $m,$ il existe $\{x_v\}\in\prod_{v\in\Omega}Z_0(X_v)$ de degré $0$ et il existe
$y\in Z_0(C)$ de degré $\delta$ à support en dehors de $D$ tel que pour toute $v\in\Omega$ on ait
$z'_v=z_v+mx_v\in Z_0(X_v/y).$
\end{lem}

\begin{proof}
Il existe toujours un zéro-cycle de degré positif sur la $k(C)$-variété $X_\eta,$ on note son degré $n$.
On fixe un nombre entier positif $a.$
Comme $\{\pi_*(z_v)\}\bot Br(C),$ d'après le théorème de Colliot-Thélène \ref{CTresult},
on trouve $y\in Z_0(C)$ de degré $\delta$ et $t_v\in Z_0(C_v)$ de degré $0$ tels que
$y\sim \pi_*(z_v)+mant_v\in Z_0(C_v).$

On peut supposer que $y$ et $t_v$ ont leurs supports en dehors de $D$ d'après le lemme de déplacement \ref{deplacement}.

Il existe un ouvert non vide $U$ de $C$ tel que pour tout $Q\in U,$
il existe un zéro-cycle de degré $n$ sur $X_Q/k(Q).$
On peut supposer aussi que les $t_v$ ont leurs supports dans $U$ par le lemme \ref{deplacement},
il existe alors, pour toute $v\in\Omega,$
$x'_v\in Z_0(X_v)$ tel que $\pi_*(x'_v)=nt_v.$ On pose $x_v=ax'_v$ et $z'_v=z_v+mx_v,$ alors
$\pi_*(z'_v)=\pi_*(z_v)+mant_v\sim y\in Z_0(C_v),$ \textit{i.e.} $z'_v\in Z_0(X_v/y).$
\end{proof}

\begin{proof}[du théorème \ref{thm4}]
On note $D$ l'ensemble fini des points fermés $Q$ de $C$ tel que la fibre $X_Q$
ne soit pas lisse.
Soit $z_v\in Z_0(X_v)$ de degré $\delta$ tels que $\{z_v\}\bot \pi^*Br(C),$ on peut supposer que
les supports de $\pi_*(z_v)$ sont disjoints de $D$ d'après
le lemme de déplacement \ref{deplacement}.

Pour $m$ un entier strictement positif fixé, le lemme \ref{lem3}  donne $x_v\in Z_0(X_v)$
de degré $0$ et $y\in Z_0(C)$ de degré $\delta$
à support en dehors de $D$ tels que pour toute $v\in \Omega$ on ait $z'_v=z_v+mx_v\in Z_0(X_v/y).$

La proposition \ref{prop1} (1), appliquée à $\{z'_v\}$ avec
$S=\{v\}$ une place non archimédienne, dit qu'il existe $z\in Z_0(X)$ tel que
$\pi_*(z)$ soit suffisamment proche de $\pi_*(\tau_v)-(s+1)y_0$ et
tel que $\pi_*(\tau_v)-(s+1)y_0 \sim y\in Z_0(C_{v}).$ Alors $deg(z)=deg(\pi_*(z))=deg(\pi_*(\tau_v)-(s+1)y_0)=deg(y)=\delta.$
\end{proof}

\begin{proof}[du théorème \ref{thm5}]\ \\
\indent(1)Grâce à la dualité locale entre la cohomologie $H^{2d}(X_v,\mathbb{Z}/m\mathbb{Z}(d))$ et la cohomologie
$H^2(X_v,\mathbb{Z}/m\mathbb{Z}(1))$ pour $v\in\Omega^\textmd{f}$
(\textit{cf.} \cite{Saito0} 2.9) et le diagramme commutatif pour $v\in\Omega^\textmd{f}$
\SelectTips{eu}{12}$$\xymatrix@C=20pt @R=14pt{
CH_0(X_v)\ar[d]^{\phi_{m,v}}\ar[r]&Hom({_mBr(X_v)},\mathbb{Q}/\mathbb{Z})\ar@{^(->}[d]\\
H^{2d}(X_v,\mathbb{Z}/m\mathbb{Z}(d))\ar[r]^{\simeq}&H^2(X_v,\mathbb{Z}/m\mathbb{Z}(1))^*,
}$$ où $d=dim(X),$
il suffit de montrer l'assertion suivante:
\smallskip

Étant donnée une famille  $\{z_v\}$ de zéro-cycles de degré
$\delta$ telle que $\{z_v\}\bot \pi^*Br(C),$ pour tout entier $m$ et
$S\subseteq\Omega^\textmd{f}$ un ensemble fini de places
non archimédiennes, il existe $z=z_{m,S}\in Z_0(X)$ de degré $\delta$
tel que $ \langle z,b \rangle _{X_v}= \langle z_v,b \rangle _{X_v}$ pour tout $b\in {_mBr(X_v)}$ et toute
$v\in S.$\bigskip

On note $D$ l'ensemble fini des points fermés $Q$ de $C$ tel que la fibre $X_Q$
ne soit pas lisse.

Soit $z_v\in Z_0(X_v)$ tel que $\{z_v\}\bot\pi^*Br(C),$ on a alors $\{\pi_*(z_v)\}\bot Br(C).$
Par le lemme de déplacement \ref{deplacement}, on peut supposer que les $\pi_*(z_v)$ sont à supports en dehors
de $D$ pour toute $v\in \Omega.$

Le lemme \ref{lem3}  donne $x_v,z'_v\in Z_0(X_v)$ de degrés $0$
et $\delta$ respectivement et $y\in Z_0(C)$ de degré $\delta$ à support en
dehors de $D,$ tels que $z'_v=z_v+mx_v\in Z_0(X/y).$
En particulier, pour tout $b\in
{_mBr(X_v)}$ et toute $v\in\Omega^\textmd{f}$ on a
$ \langle z'_v,b \rangle _{X_v}= \langle z_v,b \rangle _{X_v}.$

- Si $X_\theta$ satisfait l'approximation faible pour les points rationnels
pour tout $\theta\in \textsf{Hil},$ la proposition \ref{prop1}(2)  donne $z\in Z_0(X)$ tel que, pour $v\in S,$
$z$ soit suffisamment proche de $\tau_v-(s+1)z_0$ et $\tau_v-(s+1)z_0\sim z'_v\in Z_0(X_v),$
d'où $\pi_*(z)$ est suffisamment proche de $\pi_*(\tau_v)-(s+1)y_0$ et
$\pi_*(\tau_v)-(s+1)y_0 \sim \pi_*(z'_v)\sim y\in Z_0(C_v),$
en particulier,
$deg(z)=deg(z'_v)=\delta$ et $ \langle z,b \rangle _{X_v}= \langle z'_v,b \rangle _{X_v}= \langle z_v,b \rangle _{X_v}$
pour tout $b\in {_mBr(X_v)}$
et toute $v\in\Omega^\textmd{f}.$

- On suppose que $X_\theta$ satisfait l'approximation faible au niveau de la cohomologie aux places finies
pour les zéro-cycles de degré $1$
pour tout $\theta\in \textsf{Hil}.$ On peut supposer que $\textsf{Hil}$ et $D$ sont disjoints.

Pour un point fermé $P$ de $C$ tel que $X_P$ soit lisse, $X_P$ est alors de dimension $d-1,$
on pose $K=k(P).$
Le zéro-cycle $P_v$ défini comme l'image de $P\in Z_0(C)\buildrel{i_v^*}\over\To Z_0(C_v)$ est séparable,
où $i_v:C_v\To C.$ On sait que $i_v^{-1}(P)=Spec(K)\times_kk_v=\bigsqcup_{w|v,w\in\Omega_K} Spec(K_w),$
$P_v=i_v^*(P)=\sum_{w|v,w\in\Omega_K}P_w$ où $P_w=Spec(K_w)$ est un point fermé de $C_v$ de corps résiduel
$K_w.$ On a aussi $X_P\times_kk_v=\bigsqcup_{w|v,w\in\Omega_K}X_{P,w},$ où $X_{P,w}=X_P\times_KK_w\simeq {X_v}_{P_w}$ est
la fibre de $X_v\To C_v$ au point fermé $P_w.$
On pose $CH_0(X_P\times_kk_v)=\prod_{w|v,w\in\Omega_K}CH_0({X_v}_{P_w}).$ On trouve
que $\tilde{H}^{2(d-1)}(X_P\times_kk_v,\frac{\mathbb{Z}}{m\mathbb{Z}}(d-1))\simeq
\prod_{w|v,w\in\Omega_K}\tilde{H}^{2(d-1)}({X_v}_{P_w},\frac{\mathbb{Z}}{m\mathbb{Z}}(d-1)).$
On considère le diagramme commutatif suivant
\small
\SelectTips{eu}{10}$$\xymatrix@C=-32pt @R=30pt{
CH_0(X_P) \ar[dd]^{\phi_{P,m}}\ar[rd]^{\alpha_P}\ar[rr]^{\psi_P}&&**[r]\prod_{v\in\Omega^\textmd{f}}CH_0(X_P\times_kk_v)\ar'[d][dd]^{\phi_{P,m,v}}\ar[rd]^{\alpha_{P,v}}
\\&CH_0(X)\ar[dd]^(.7){\phi_m}\ar[rr]^(.3)\psi&&\prod_{v\in\Omega^\textmd{f}}CH_0(X_v)\ar[dd]^(.7){\phi_{m,v}}
\\H^{2(d-1)}(X_P,\frac{\mathbb{Z}}{m\mathbb{Z}}(d-1))\ar'[r][rr]^{\psi^H_{mP}}\ar[rd]^{\beta_P}&&**[r]\prod'_{v\in\Omega^\textmd{f}}\tilde{H}^{2(d-1)}(X_P\times_kk_v,\frac{\mathbb{Z}}{m\mathbb{Z}}(d-1))\ar[rd]^{\beta_{P,v}}
\\&H^{2d}(X,\frac{\mathbb{Z}}{m\mathbb{Z}}(d))\ar[rr]^{\psi^H_m}&&\prod'_{v\in\Omega^\textmd{f}}\tilde{H}^{2d}(X_v,\frac{\mathbb{Z}}{m\mathbb{Z}}(d))
}$$

\normalsize

Dans ce diagramme, les $\phi$ sont des applications cycle, par exemple
$\phi_{P,m,v}=\prod_{w|v,w\in\Omega_K}\phi_{P_w}$ où
$$\phi_{P_w}:CH_0({X_v}_{P_w})\To\tilde{H}^{2(d-1)}({X_v}_{P_w},\frac{\mathbb{Z}}{m\mathbb{Z}}(d-1)).$$
L'application $\alpha_{P,v}$ est définie par $\alpha_{P,v}=\prod_{w|v,w\in\Omega_K}\alpha_{P,w},$ où
$$\alpha_{P,w}:CH_0({X_v}_{P_w})\To CH_0(X_v).$$
L'application $\beta_{P,v}$ est définie par $\beta_{P,v}=\prod_{w|v,w\in\Omega_K}\beta_{P,w},$ où
$$\beta_{P,w}:\tilde{H}^{2(d-1)}({X_v}_{P_w},\frac{\mathbb{Z}}{m\mathbb{Z}}(d-1))\To \tilde{H}^{2d}(X_v,\frac{\mathbb{Z}}{m\mathbb{Z}}(d))$$
est l'application de Gysin pour le $k_v$-couple lisse $({X_v}_{P_w},X_v)$,
qui est bien définie au moins pour une place non archimédienne
(\textit{cf.} \cite{Thomason} et \cite{SGA4} XVI \S 3). L'application $\beta_P$ est l'application de Gysin pour le $k$-couple lisse
$(X_P,X)$ de codimension $1.$ Les $\psi$ sont des applications de type local-global.

Comme les applications cycle commutent avec les applications de Gysin (\textit{cf.} \cite{MilneEC} VI.9.3),
le diagramme ci-dessus est commutatif.

La proposition \ref{prop1}  donne un point fermé $\theta\in \textsf{Hil},$ vu comme zéro-cycle
effectif séparable ${\theta}\in Z_0(C)$
qui est suffisamment proche de $\pi_*(\tau_v)$ pour toute $v\in S\subset\Omega^\textmd{f}.$
On pose $P=\theta,$ c'est un point fermé de corps résiduel $K=k(\theta).$
On pose $P_v=i^*_v(P)=\sum_{w|v,w\in\Omega_K}P_w\in Z_0(C_v)$ où le corps résiduel $k(P_w)=K_w.$
Chaque $P_w$ est suffisamment proche d'un point fermé (unique) $Q_w$ (de corps résiduel $K_w$) de $C_v$ qui apparaît dans
$\pi_*(\tau_v).$ Comme $\pi_*(\tau_v)$ est séparable,
il y a exactement un point fermé $R_w$ (de corps résiduel $K_w$) de $X_v,$ qui apparaît dans
$\tau_v$ contenu dans la fibre ${X_v}_{Q_w}.$
Le théorème des fonctions implicites  donne un point fermé $T_w$ (de corps résiduel $K_w$) de $X_v$ contenu dans la fibre
${X_v}_{P_w},$ suffisamment proche de $R_w,$ tel que $ \langle T_w,b \rangle = \langle R_w,b \rangle $
pour tout $b\in{_mBr(X_v)},$ autrement dit,
$T_w$ et $R_w$ ont la même image dans $\tilde{H}^{2d}(X_v,\frac{\mathbb{Z}}{m\mathbb{Z}}(d))$, \textit{i.e.}
$\phi_{m,v}(\alpha_{P,w}([T_w]))=\phi_{m,v}([R_w])$ où $[-]$ dénote la classe d'un zéro-cycle dans le groupe de Chow.
Par hypothèse, il existe un zéro-cycle $P'\in Z_0(X_P)$ de degré $1$
tel que, pour toute $w|v,$ $v\in S\subset\Omega^\textmd{f},$ et tout
$b\in{_mBr(X_{P,w})}$ (on rappelle que $X_{P,w}={X_v}_{P_w}$), on ait $ \langle P',b \rangle = \langle T_w,b \rangle .$ Autrement dit,
$P'$ et $T_w$ ont même image dans $\tilde{H}^{2(d-1)}({X_v}_{P_w},\frac{\mathbb{Z}}{m\mathbb{Z}}(d-1)),$
\textit{i.e.} $\phi_{P_w}(\psi_P([P']))=\phi_{P_w}([T_w]).$ Grâce au diagramme commutatif ci-dessus,
$\alpha_P([P'])\in CH_0(X)$ et $R_w$ ont donc même
image dans $\tilde{H}^{2d}(X_v,\frac{\mathbb{Z}}{m\mathbb{Z}}(d))$ pour toute $v\in S.$
Avec les notations de la proposition \ref{prop1}, on pose $z'=P'$
et $z=z'-(s+1)z_0\in Z_0(X),$ avec un calcul simple, on trouve finalement
que $z$ et $\{z_v\}$ ont la même image dans $\prod_{v\in S}\tilde{H}^{2d}(X_v,\frac{\mathbb{Z}}{m\mathbb{Z}}(d)),$
ceci complète la preuve de (1).

(2)De plus, avec l'hypothèse (H CH0), on peut supposer que pour toute $v\notin S$ l'application $CH_0(X_v)\To CH_0(C_v)$
est injective. D'après l'argument ci-dessus, il reste à vérifier que pour $v\in \Omega^\textmd{f}\setminus S$ on a:
$z$ et $z_v$ ont même image dans $\tilde{H}^{2d}(X_v,\mathbb{Z}/m\mathbb{Z}(d)),$
pour le zéro-cycle global $z\in Z_0(X)$ obtenu dans (1).
Avec l'injectivité de l'application $CH_0(X_v)\To CH_0(C_v),$ il suffit de vérifier que
$z$ et $z_v$ ont même image dans $CH_0(C_v).$
Or, avec les notations dans (1) et dans la proposition \ref{prop1}, on sait alors que $z=z'-(s+1)z_0$ et
$\pi_*(z)=\pi_*(z')-(s+1)y_0={\theta}-(s+1)y_0\sim\pi_*(z_v)$ pour toute $v\in\Omega.$
\end{proof}

\begin{rem}
Dans le théorème \ref{thm5}, en comparant avec le cas où $X$ est une courbe,
on ne peut rien dire autour des places réelles, les difficultés sont les suivantes.

(1)Est-ce que la surjection $H^2(X_\mathbb{R},\mathbb{Z}/m\mathbb{Z}(1))\twoheadrightarrow{_mBr(X_\mathbb{R})}$
se factorise à travers $H^2(X_\mathbb{R},\mathbb{Z}/m\mathbb{Z}(1))\To \tilde{H}^2(X_\mathbb{R},\mathbb{Z}/m\mathbb{Z}(1))$?
(La réponse à la même question sur $\mathbb{C}$ est négative.)

(2)Existe-t-il une application de Gysin
$$\tilde{H}^{2(d-1)}(X_{P,\mathbb{R}},\mathbb{Z}/m\mathbb{Z}(d-1))\To \tilde{H}^{2d}(X_{\mathbb{R}},\mathbb{Z}/m\mathbb{Z}(d))$$
pour la paire lisse $(X_{P,\mathbb{R}},X_\mathbb{R}),$
commutant avec les applications cycle?
\end{rem}

\begin{rem} Dans le théorème \ref{thm5},
on part d'une hypothèse arithmétique sur les fibres au niveau de
la cohomologie, et on obtient une conclusion sur l'espace total au niveau de la cohomologie. Sans modifier les détails des preuves,
une fois qu'on a établi l'énoncé (E) suivant, on peut partir d'une hypothèse d'approximation faible au niveau du groupe de Chow
et obtenir une conclusion d'approximation faible/forte au niveau du groupe de Chow sur l'espace total.
L'énoncé (E) est montré par Wittenberg dans \cite{Wittenberg}, lemme 1.8.\ \\
(E)``proche $\Rightarrow$ équivalence rationnelle modulo $m$": Soit $X$ une variété projective lisse sur
un corps local de caractéristique $0,$
et soit $m$ un entier strictement positif. Alors deux zéro-cycles effectifs $z,z'\in Z_0(X)$ sont rationnellement
équivalents modulo $m$ (\textit{i.e.} ils ont même image dans $CH_0(X)/m$),
lorsque $z$ est suffisamment proche de $z'.$
\end{rem}


\bibliographystyle{plain}
\bibliography{mybib1}
\end{document}